\documentclass[a4paper,11pt]{amsart}

\usepackage{amsmath, amssymb, amsthm}

\newcommand{\sP}{\mathsf{P}}
\newcommand{\bP}{\mathbb{P}}
\newcommand{\bN}{\mathbb{N}}

\newcommand{\cO}{\mathcal{O}}
\newcommand{\cS}{\mathcal{S}}
\newcommand{\sA}{\mathsf{A}}
\newcommand{\sH}{\mathsf{H}}
\newcommand{\sF}{\mathsf{F}}
\newcommand{\sJ}{\mathsf{J}}
\newcommand{\diam}{\operatorname{diam}}

\newcommand{\PGL}{\mathrm{PGL}}
\newcommand{\supp}{\operatorname{supp}}
\newcommand{\bR}{\mathbb{R}}
\newcommand{\Res}{\operatorname{Res}}

\newcommand{\Capa}{\operatorname{Cap}}
\newcommand{\cM}{\mathcal{M}}
\newcommand{\cF}{\mathcal{F}}

\newcommand{\sB}{\mathsf{B}}

\numberwithin{equation}{section}
\theoremstyle{plain}
\newtheorem{theorem}{Theorem}[section]

\newtheorem{mainth}{Theorem}

\theoremstyle{definition}
\newtheorem{definition}[theorem]{Definition}

\newtheorem*{acknowledgement}{Acknowledgement}
\theoremstyle{remark}
\newtheorem{remark}[theorem]{Remark}

\begin{document}
\title[Uniform perfectness of the Berkovich Julia sets]{
Uniform perfectness of the Berkovich Julia sets
in non-archimedean dynamics}
\author[Y\^usuke Okuyama]{Y\^usuke Okuyama
}
\address{
Division of Mathematics,
Kyoto Institute of Technology, Sakyo-ku,
Kyoto 606-8585 Japan}
\email{okuyama@kit.ac.jp}

\date{\today}


\begin{abstract}
 We show that a rational function $f$ of degree $>1$ 
 on the projective line over an algebraically closed field that is complete 
 with respect to a non-trivial and non-archimedean absolute value
 has no potentially good reductions if and only if
 the Berkovich Julia set of $f$ is uniformly perfect. As an application,
 a uniform regularity of the boundary of each Berkovich Fatou component of $f$
 is also established.
\end{abstract}

\maketitle

\section{Introduction}\label{sec:intro}

Let $K$ (or $(K,|\cdot|)$) be an algebraically closed field 
(of arbitrary characteristic) that is complete
with respect to a non-trivial and non-archimedean absolute value $|\cdot|$.
A subset $B$ in $K$ is called a $K$-closed disk if 
\begin{gather*}
 B=B(a,r):=\{z\in K:|z-a|\le r\}
\end{gather*}
for some $a\in K$ and some $r\ge 0$; by the strong triangle inequality
$|z+w|\le\max\{|z|,|w|\}$ for any $z,w\in K$,
the diameter $\diam(B(a,r))$ of $B(a,r)$ in $(K,|\cdot|)$
equals $r$, and we have $B(a,r)=B(b,r)$ for any $b\in B(a,r)$.
When $\diam B>0$, 
a $K$-closed disk $B$ is in fact clopen as a subset in $(K,|\cdot|)$.

By the strong triangle inequality again,
for any two $K$-closed disks $B,B'$, if $B\cap B'\neq\emptyset$,
then either $B\subset B'$ or $B\supset B'$. 

\subsection{Berkovich projective line}
The Berkovich projective line $\sP^1=\sP^1(K)$ is 
the Berkovich analytification (a compact augmentation) 
of the (classical) projective line 
$\bP^1=\bP^1(K)=K\cup\{\infty\}$ (see the foundational Berkovich \cite{Berkovichbook}).
As a set, $\sP^1$ is almost identified with the set of all $K$-closed disks; 
each singleton $\{a\}=B(a,0)$ in $K=\bP^1\setminus\{\infty\}$ is identified with the type I point $a\in K\subset\sP^1$, and each $K$-closed disk $B$ satisfying $\diam B\in|K^*|$ (resp.\ $\diam B\in\bR_{>0}\setminus |K^*|$) is identified with a type II (resp.\ type III) point $\cS$ in $\sP^1$ (then we also write $B=B_{\cS}$ and $\cS=\cS_B$). The remaining point $\infty\in\bP^1\subset\sP^1$ is also a type I point in $\sP^1$, and the set of all other remaining points in $\sP^1$
is exactly the set $\sH^1_{\mathrm{IV}}$ of all type IV points in 
$\sP^1$.
The Berkovich upper half space
is
\begin{gather*}
 \sH^1=\sH^1(K):=\sP^1\setminus\bP^1=\sH^1_{\mathrm{II}}\cup\sH^1_{\mathrm{III}}\cup\sH^1_{\mathrm{IV}}\quad(\text{a disjoint union}),
\end{gather*}
where $\sH^1_{\mathrm{II}}$ (resp.\ $\sH^1_{\mathrm{III}}$) is
the set of all type II (resp.\ type III) points in $\sP^1$.
By the above alternative for two intersecting $K$-closed disks, the inclusion relation among $K$-closed disks extends to a (partial) ordering $\prec$ on $\sP^1$ 
so that for any $\cS,\cS'\in\sP^1\setminus(\{\infty\}\cup\sH^1_{\mathrm{IV}})$, $\cS\prec\cS'$ if and only if $B_{\cS}\subsetneq B_{\cS'}$, 
that $\infty$ is the unique maximum element in $(\sP^1,\prec)$, 
that $K\cup\sH^1_{\mathrm{IV}}$ coincides with
the set of all minimal elements in $(\sP^1,\prec)$,
and that for any $\cS,\cS'\in\sP^1$, 
there is the minimum element $\cS''$ in $(\sP^1,\prec)$ 
satisfying both $\cS\preceq\cS'$ and $\cS'\preceq\cS''$, 
which we denote by $\cS\wedge\cS'$.
For every $\cS,\cS'\in\sP^1$, the (closed) interval $[\cS,\cS']$ is defined as
$[\cS,\cS']=\{\cS''\in\sP^1:\cS\preceq\cS''\preceq\cS'\}$ if $\cS\preceq\cS'$,
and is defined as 
\begin{gather*}
 [\cS,\cS']=[\cS,\cS\wedge\cS']\cup[\cS\wedge\cS',\cS'] 
\end{gather*}
in general. For every $\cS\in\sP^1$, the directions space $T_{\cS}\sP^1$ of $\sP^1$ at $\cS$
is the set of all germs $\overrightarrow{\cS\cS'}$
of (left half open) intervals $(\cS,\cS']=[\cS,\cS']\setminus\{\cS\}$, $\cS'\in\sP^1\setminus\{\cS\}$. Setting
\begin{gather*}
 U_{\overrightarrow{v}}=U_{\cS,\overrightarrow{v}}:=\bigl\{\cS'\in\sP^1\setminus\{\cS\};\overrightarrow{\cS\cS'}=\overrightarrow{v}\bigr\}\quad\text{for each }\overrightarrow{v}\in T_{\cS}\sP^1,
\end{gather*}
the family $\{U_{\cS,\overrightarrow{v}}:\cS\in\sH^1_{\mathrm{II}},\overrightarrow{v}\in T_{\cS}\sP^1\}$ is a quasi open basis of the (so called weak) topology of $\sP^1$; $\sP^1$ is indeed a compact, connected, uniquely arcwise connected, and Hausdorff topological space.
For every $\cS\in\sP^1$ and every $\overrightarrow{v}\in T_{\cS}\sP^1$,
$U_{\overrightarrow{v}}$ is a component of $\sP^1\setminus\{\cS\}$
(and $\partial(U_{\overrightarrow{v}})=\{\cS\}$), and
any component of $\sP^1\setminus\{\cS\}$ is written as
$U_{\overrightarrow{v}}$ by a unique $\overrightarrow{v}\in T_{\cS}\sP^1$.
Both $\bP^1$ and $\sH^1_{\mathrm{II}}$ are (non-empty and) dense in $\sP^1$, 
and so are $\sH^1_{\mathrm{III}}$ and $\sH^1_{\mathrm{IV}}$
if they are non-empty. 

The (affine) diameter function $\diam:\cS\mapsto\diam(B_{\cS})$ on $\sP^1\setminus(\{\infty\}\cup\sH^1_{\mathrm{IV}})$ 
extends to an upper semicontinuous function $\sP^1\to[0,+\infty]$
so that $\diam(\infty)=+\infty$, that $\sH^1=\{\cS\in\sP^1:\diam\cS\in(0,+\infty)\}$, and that the restriction $\diam|[\cS,\cS']$ is continuous for any $\cS,\cS'\in\sP^1$. For references and more details, see Section \ref{sec:background}.

\subsection{Uniformly perfect subsets in $\sP^1$}
The Berkovich affine line is
\begin{gather*}
 \sA^1=\sA^1(K)=\sP^1\setminus\{\infty\}=U_{\overrightarrow{\infty 0}}. 
\end{gather*}
Let us call an open subset $A$ in $\sP^1$ a (non-degenerating and finite
open Berkovich) concentric annulus in $\sA^1$
if $A=A(\cS,\cS'):=U_{\overrightarrow{\cS\cS'}}\cap U_{\overrightarrow{\cS'\cS}}$ 
for some $\cS,\cS'\in\sH^1$ satisfying $\cS\prec\cS'$, and
the modulus of $A$ is defined by
\begin{gather*}
 \operatorname{mod}A:=\log\bigl(\diam(\cS')/\diam\cS\bigr)\in\bR_{>0}.
\end{gather*}
For a subset $E$ in $\sP^1$, we say a concentric annulus 
$A=A(\cS,\cS')$ in $\sA^1$ separates $E$ if 
\begin{gather*}
A\subset\sP^1\setminus E,\quad
E\cap(\sP^1\setminus U_{\overrightarrow{\cS\cS'}})\neq\emptyset,
\quad\text{and}\quad
E\cap(\sP^1\setminus U_{\overrightarrow{\cS'\cS}})\neq\emptyset,
\end{gather*}
and say $E$ has the bounded moduli property for its separating annuli if
\begin{gather}
 \sup\bigl\{\operatorname{mod}A:A\text{ is a concentric annulus in }\sA^1
 \text{ separating }E\bigr\}<+\infty,\label{eq:boundedmoduli}
\end{gather}
adopting the convention $\sup\emptyset=0$. 

See \eqref{eq:upfree} for a coordinate-free formulation of \eqref{eq:boundedmoduli}. In non-archimedean setting, the condition \eqref{eq:boundedmoduli} for 
a non-empty compact subset $E$ in $\sP^1$
itself does not imply the perfectness of $E$ (since $\diam\cS>0$ for $\cS\in\sH^1$).
The following stronger notion than that of a perfect subset in $\sP^1$
is a non-archimedean counterpart to that which has been well studied
in complex analysis/dynamics (see, e.g, the survey \cite{sugawa03}).

\begin{definition}\label{th:up}
A subset $E$ in $\sP^1$ is uniformly perfect 
if $E$ not only is perfect, i.e., is non-empty and compact and
has no isolated points in $\sP^1$, but also
has the bounded moduli property for its separating annuli.
\end{definition}

In Section \ref{sec:lcd}, we see that the bounded moduli property of a 
compact subset $E$ in $\sP^1$ for its separating annuli is equivalent 
to a potential theoretic (non-infinitesimal) lower capacity density property of $E$
(see Theorem \ref{th:lcd} in Section \ref{sec:lcd}), and then
in Section \ref{sec:holder}, we see that those equivalent properties imply 
a H\"older continuity (uniformly at any point of $(\partial D)\cap\bP^1$)
of the logarithmic potential $p_{\cS_0,\nu_{\cS_0,E}}$ on $\sP^1$
of the (unique) equilibrium mass distribution $\nu_{\cS_0,E}$ on $E$ 
with respect to any point $\cS_0$ in a component $D$ of $\sP^1\setminus E$ 
or equivalently, that of the generalized Green function $G_{\cS_0,E}$ 
on $\sP^1$ associated to $E$ with respect to $\cS_0\in D$ (see Theorem \ref{th:holder} in Section \ref{sec:holder}).

\subsection{Non-archimedean dynamics and the principal results}
The analytic action on $\bP^1$ of a rational function $h\in K(z)$
extends analytically (so continuously) to that on $\sP^1$.
If in addition $\deg h>0$, then this extended analytic action on $\sP^1$ of $h$ 
preserves the types I, II, III, or IV of points in $\sP^1$ and is surjective,
open, and fiber-discrete.

For a rational function $f\in K(z)$ of degree $d>1$, the Berkovich Julia set
$\sJ(f)$ is the set of all points $\cS\in\sP^1$ such that 
\begin{gather*}
\bigcap_{U:\text{ open in }\sP^1\text{ containing }\cS}\biggl(\bigcup_{n\in\bN}f^n(U)\biggr)=\sP^1\setminus E(f),
\end{gather*}
where the classical exceptional set
$E(f):=\{a\in\bP^1:\#\bigcup_{n\in\bN}f^{-n}(a)<\infty\}$ of $f$ is at most countable (and $\#E(f)\le 2$ if in addition $K$ has characteristic $0$),
and the Berkovich Fatou set $\sF(f)$ is defined by $\sP^1\setminus\sJ(f)$,
so that $\sJ(f)$ and $\sF(f)$ are respectively compact and open in $\sP^1$. 
In fact, $\sJ(f)$ is a non-empty compact 
(and nowhere dense) subset in $\sP^1$, and 
both $\sJ(f)$ and $\sF(f)$ are totally invariant under $f$ in that
\begin{gather*}
 f^{-1}(\sJ(f))=\sJ(f) \quad\text{and}\quad f^{-1}(\sF(f))=\sF(f).
\end{gather*}
Identifying $\PGL(2,K)$ with
the linear fractional transformations group on $\bP^1$,
both $\sJ(f)$ and $\sF(f)$ are
also $\PGL(2,K)$-equivariant in that for every $h\in\PGL(2,K)$,
\begin{gather*}
 \sJ(h\circ f\circ h^{-1})=h(\sJ(f))\quad\text{and}\quad\sF(h\circ f\circ h^{-1})=h(\sF(f)).
\end{gather*}
For references, see Section \ref{sec:background}.

We say $f$ has a potentially good reduction if
$f^{-1}(\cS)=\{\cS\}$ for some $\cS\in\sH^1$
(and then $\sJ(f)=\{\cS\}$),
and otherwise, we say $f$ has no potentially good reductions. Indeed,
$\sJ(f)$ has no isolated points (so $\sJ(f)$ is perfect)
if and only if $f$ has no potentially good reductions. 

Our first principal result is the following non-archimedean counterpart to
Eremenko \cite{EremenkoUP}, Hinkkanen \cite{Hinkkanen93}, 
and Ma\~n\'e--da Rocha \cite{MR92} in complex dynamics. 

\begin{mainth}\label{th:upJulia}
Let $K$ be an algebraically closed field that is complete
with respect to a non-trivial and non-archimedean absolute value. Then for every rational function $f\in K(z)$ of degree $>1$,  
$f$ has no potentially good reductions if and only if
the Berkovich Julia set $\sJ(f)$ is uniformly perfect.
\end{mainth}

In particular, if $f$ has no potentially good reductions,
then $\sJ(f)$ has the (non-infinitesimal) lower capacity density property,
as mentioned above. Theorem \ref{th:upJulia} is shown in Section \ref{sec:upJulia}.

A component of $\sF(f)$ is called a Berkovich Fatou component of $f$. 
Our second principal result
is the following uniform regularity of the boundary of 
any Berkovich Fatou component of $f$.


\begin{mainth}\label{th:hrFatou}
Let $K$ be an algebraically closed field that is complete
with respect to a non-trivial and non-archimedean absolute value,
and let $f\in K(z)$ be a rational function on $\bP^1$ of degree $>1$.  
Then for every Berkovich Fatou
component $D$ of $f$ and every $\cS_0\in D$, 
there are $\alpha>0$ and $\delta_0\in(0,1]$ such that 
\begin{gather*}
 \sup_{a\in(\partial(D_{\cS_0}))\cap\bP^1}\sup_{\delta\in(0,\delta_0]}
\frac{\sup_{\sB_\#(a,\delta)}
\bigl(p_{\cS_0,\nu_{\cS_0,\sP^1\setminus D}}(\cdot)-\log\Capa_{\cS_0}(\sP^1\setminus D)\bigr)}{\delta^\alpha}<+\infty.
\end{gather*}
Here $\Capa_{\cS_0}E$ is the logarithmic capacity of a subset
$E$ in $\sP^1$ with respect to $\cS_0\in\sP^1\setminus E$ $($see \eqref{eq:capacitygeneral}$)$, and $\sB_\#(a,\delta)$
is the Berkovich chordal closed ball in $\sP^1$
centered at $a\in\bP^1$ and of diameter $\delta\in(0,1]$ $($see $\S\ref{th:kernel})$.
\end{mainth}

The assertion in Theorem \ref{th:hrFatou} 
is clear if $f$ has a potentially good reduction, and
follows from Theorem \ref{th:upJulia} together with
Theorems \ref{th:lcd} and \ref{th:holder} below
if $f$ has no potentially good reductions.
Theorem \ref{th:hrFatou} simplifies our former computations 
in \cite[\S 4]{OS15}
of $p_{\infty,\nu_{\infty,\sP^1\setminus D_\infty}}$, letting
$D_\infty$ be the Berkovich Fatou component of $f$ containing $\infty$
under the normalization that $\infty\in\sF(f)$. This simplification
was one of our aims in this paper.
 
\section{Background}\label{sec:background}

For the foundation on non-archimedean dynamics on $\sP^1=\sP^1(K)$, see 
\cite{BR10,BenedettoBook,FJbook,FR06,ThuillierThesis},
and the survey \cite[\S 1-\S 4]{Jonsson15}.
In what follows, we adopt notations from \cite[\S 2, \S 3]{OkuDivisor}.   

\subsection{$\sH^1$ as a hyperbolic space}
We begin with continuing the exposition on $\sP^1$ in Section \ref{sec:intro}.
For every $a\in\bP^1$, $\#T_a\sP^1=1$. On the other hand,
\begin{gather*}
\sH^1_{\mathrm{IV}}=\{\cS\in\sH^1;\#T_\cS\sP^1=1\}\quad\text{and}\quad 
\sH^1_{\mathrm{III}}=\{\cS\in\sP^1;\#T_\cS\sP^1=2\}. 
\end{gather*}
For each $a\in\bP^1(K)$, writing $a=[a_0:a_1]$, where $\max\{|a_0|,|a_1|\}=1$,
the point $\tilde{a}:=[\tilde{a_0}:\tilde{a_1}]\in\bP^1(k)$ is called 
the reduction of $a$ modulo $\cM_K$, where 
$\cO_K=B(0,1)$ is the ring of $K$-integers, 
\begin{gather*}
 \cM_K=\{z\in K:|z|<1\}(=U_{\overrightarrow{\cS_g0}}\cap K) 
\end{gather*}
is the unique maximal ideal of $\cO_K$, and
$k=k_K=\cO_K/\cM_K$ is the residue field of $K$. 
For every $\cS\in\sH^1_{\mathrm{II}}$, there is a bijection between $T_{\cS}\sP^1$ and $\bP^1(k)$, so that $\#T_{\cS}\sP^1=\#\bP^1(k)>2$; for the Gauss
(or canonical) point
\begin{gather*}
 \cS_g:=\cS_{\cO_K}
\end{gather*}
($\diam(\cS_g)=1\in|K^*|$ so $\cS_g\in\sH^1_{\mathrm{II}}$) in $\sP^1$, 
there is a canonical bijection 
$T_{\cS_g}\sP^1\ni\overrightarrow{v}\leftrightarrow\tilde{a}\in\bP^1(k)$, where
$a$ is any point in $U_{\overrightarrow{v}}\cap\bP^1(K)$.

The hyperbolic (or big model) metric $\rho$ on $\sH^1$ is defined as
\begin{align*}
 \rho(\cS,\cS'):=&2\log\diam(\cS\wedge\cS')-\log\diam\cS-\log\diam(\cS')\\
=&\begin{cases}
		  \log\bigl(\diam(\cS')/\diam\cS\bigr)\quad(\cS,\cS'\in\sH^1,\cS\preceq\cS'),\\
		  \rho(\cS,\cS\wedge\cS')+\rho(\cS\wedge\cS',\cS)\quad(\cS,\cS'\in\sH^1,\text{ in general}).
		 \end{cases}
\end{align*}
The Berkovich hyperbolic space $(\sH^1,\rho)$ is an $\bR$-tree having
$\bP^1$ as the ideal (or Gromov) boundary, and
the topology of $(\sH^1,\rho)$ is finer than the relative topology of $\sH^1$ in $\sP^1$. 
For every $\cS\in\sP^1$, the action on $\sP^1$ of a rational function $h\in K(z)$ 
on $\bP^1$ of degree $>0$ induces the tangent map
$h_*=(h_*)_{\cS}:T_{\cS}\sP^1\to T_{h(\cS)}\sP^1$ so that
for every $\overrightarrow{v}=\overrightarrow{\cS\cS'}\in T_{\cS}\sP^1$, 
if $\cS'\in\sP^1\setminus\{\cS\}$ is close enough to $\cS$, 
then for any distinct $\cS'',\cS'''\in[\cS,\cS']$, we have
$\rho(h(\cS''),h(\cS'''))/\rho(\cS,\cS')\equiv m_{\overrightarrow{v}}(h)$,
where the constant $m_{\overrightarrow{v}}(h)\in\{1,\ldots,\deg h\}$
is called the directional local degree of $h$ with respect to $\overrightarrow{v}$,
and $h_*(\overrightarrow{v})=\overrightarrow{h(\cS)h(\cS')}$; 
moreover, 
\begin{gather}
h(U_{\overrightarrow{v}})=\text{either }\sP^1\text{ or }U_{h_*(\overrightarrow{v})}.\label{eq:image} 
\end{gather}

The hyperbolic metric $\rho$ is independent of the choice of a projective coordinate on $\bP^1$;
indeed, $\rho$ is invariant under the restriction to $\sH^1$
of the (extended analytic) action on $\sP^1$
of the linear fractional transformations group $\PGL(2,K)$.
Let us call 
an open subset $A$ in $\sP^1$ a (non-degenerating and finite 
open Berkovich) annulus in 
$\sP^1$ if 
\begin{gather*}
 A=A(\cS,\cS'):=U_{\overrightarrow{\cS\cS'}}\cap U_{\overrightarrow{\cS'\cS}} 
\quad\text{for some distinct }\cS,\cS'\in\sH^1, 
\end{gather*}
and then the modulus of $A$ is defined as
\begin{gather*}
 \operatorname{mod}A:=\rho(\cS,\cS')\in\bR_{>0};
\end{gather*}
the moduli of annuli are $\PGL(2,K)$-invariant, and
for any annulus $A=A(\cS,\cS')$ in $\sP^1$ and
any $\cS''\in
[\cS,\cS']\setminus\{\cS,\cS'\}$, 
\begin{gather}
 \operatorname{mod}A
=\operatorname{mod}\bigl(A(\cS,\cS'')\bigr)+\operatorname{mod}\bigl(A(\cS',\cS'')\bigr).\label{eq:modulusdecomp}
\end{gather}
Similarly to concentric annuli, for a subset $E$ in $\sP^1$, 
we say an annulus $A=A(\cS,\cS')$ separates $E$ if $A\subset\sP^1\setminus E$,
$E\cap(\sP^1\setminus U_{\overrightarrow{\cS\cS'}})\neq\emptyset$, 
and $E\cap(\sP^1\setminus U_{\overrightarrow{\cS'\cS}})\neq\emptyset$,
and the condition \eqref{eq:boundedmoduli} for this $E$ is reformulated as
\begin{gather}
 \sup\bigl\{\operatorname{mod}A:A\text{ is an annulus in }\sP^1\text{ separating }E\bigr\}<\infty.\label{eq:upfree}
\end{gather}

\subsection{Chordal metric on $\bP^1$ and the Hsia kernels on $\sP^1$}\label{th:kernel}
Let $\pi:K^2\setminus\{(0,0)\}\to\bP^1$ 
be the canonical projection.
On $K^2$, let $\|(p_0,p_1)\|$ be the maximum norm $\max\{|p_0|,|p_1|\}$.
With the wedge product $(p_0,p_1)\wedge(q_0,q_1):=p_0q_1-p_1q_0$ on $K^2$,
the (normalized) chordal metric on $\bP^1$ is  defined as
$[z,w]:=|p\wedge q|/(\|p\|\cdot\|q\|)$
on $\bP^1\times\bP^1$, where $p\in\pi^{-1}(z),q\in\pi^{-1}(w)$, so 
that
\begin{gather*}
 [z,w]=\frac{|z-w|}{\max\{1,|z|\}\max\{1,|w|\}}\le\max\{1,|z-w|\}\quad\text{on }K\times K
\end{gather*}
and that the strong triangle inequality $[z,w]\le\max\{[z,w'],[w',w]\}$
for any $z,w,w'\in\bP^1$ still holds. The topology
of $(\bP^1,[z,w])$ coincides with the relative topology of $\bP^1$ in $\sP^1$.
Let us denote by $\diam_\#S$ 
the diameter of a (non-empty) subset $S$ in $(\bP^1,[z,w])$.


The action on $(\bP^1,[z,w])$ of a rational function $h\in K(z)$ 
of degree $>0$ is 
$1/|\Res(h)|$-Lipschitz continuous, where $\Res(h)\in[1,+\infty)$ 
is the homogeneous resultant $\Res(H_0,H_1)$ of a (non-degenerate
homogeneous) lift $H=(H_0,H_1)\in(K[X_0,X_1]_{\deg h})^2$ of $h(=[H_0:H_1])$ 
normalized so that the maximum of the absolute values of all coefficients of
$H_0,H_1$ equals $1$ (\cite[Theorem 0.1]{RumelyWinburn15}). 
In particular, the action on $(\bP^1,[z,w])$ of each element
$m(z)=(az+b)/(cz+d)$ in $\PGL(2,K)$, where $\max\{|a|,|b|,|c|,|d|\}=1$, is quasi-isometric in that
for any $z,w\in\bP^1$,
$|ad-bc|[z,w]\le[M(z),M(w)]\le |ad-bc|^{-1}[z,w]$,
so in particular the subgroup
$\PGL(2,\cO_K)$ in $\PGL(2,K)$ acts isometrically on $(\bP^1,[z,w])$.

The generalized Hsia kernel $[\cS,\cS']_g$ on $\sP^1$ with respect to $\cS_g$
is the unique upper semicontinuous and separately continuous extension to $\sP^1\times\sP^1$
of the chordal metric function $[z,w]$ on $\bP^1\times\bP^1$.
More generally, for each $\cS_0\in\sP^1$,
the generalized Hsia kernel on $\sP^1$ with respect to $\cS_0$ is defined as
\begin{gather*}
[\cS,\cS']_{\cS_0}:=[\cS,\cS']_g/([\cS,\cS_0]_g[\cS',\cS_0]_g)\quad\text{on }\sP^1\times\sP^1,
\end{gather*}
adopting the convention $1/0=0/(0^2)=+\infty$; indeed,
denoting by $\cS\wedge_{\cS_0}\cS'$
the unique point in $[\cS,\cS']\cap[\cS,\cS_0]\cap[\cS',\cS_0]$
for any $\cS,\cS',\cS_0\in\sP^1$ (so in particular
that $\cS\wedge_\infty\cS_0=\cS\wedge\cS'$),
the Gromov product on $(\sH^1,\rho)$ with respect to 
a point $\cS_0\in\sH^1$ is written as
\begin{gather}
\rho(\cS_0,\cS\wedge_{\cS_0}\cS')=-\log[\cS,\cS']_{\cS_0}+\rho(\cS_g,\cS_0)\quad\text{on }\sH^1\times\sH^1.\label{eq:Gromov} 
\end{gather}
The generalized Hsia kernel function 
$[\cS,\cS]_{\cS_0}$ not only satisfies the strong triangle inequality
\begin{gather*}
 [\cS,\cS']_{\cS_0}\le\max\bigl\{[\cS,\cS'']_{\cS_0},[\cS'',\cS']_{\cS_0}\bigr\}
\quad\text{for any }\cS,\cS'\in\sP^1
\end{gather*}
but also is $\PGL(2,\cO_K)$-invariant in that
for every $M\in\PGL(2,\cO_K)$, 
\begin{gather}
 [M(\cS),M(\cS')]_{M(\cS_0)}=[\cS,\cS']_{\cS_0}\quad\text{on }\sP^1\times\sP^1.\label{eq:Hsiainvariant}
\end{gather}
Let us denote $[\cS,\cS']_\infty$ also by $|\cS-\cS'|_\infty$, 
(the restriction to $\sA^1\times\sA^1$ of) which is the original Hsia kernel 
on $\sA^1$
and restricts to $|z-w|$ on $K\times K$, although $\cS-\cS'$ itself is not necessarily defined. For any $\cS,\cS'\in\sA^1$, 
\begin{multline}
 |\cS-\cS'|_\infty=
\bigl\{|z-z'|:z\in B_{\cS},z'\in B_{\cS'}\bigr\}=\diam(\cS\wedge\cS')\\
\ge\max\{\diam\cS,\diam(\cS')\}.\label{eq:Hsiadiam}
\end{multline}
For every (non-empty) subset $S$ in $\sA^1$, we also set 
\begin{gather*}
 \diam_\infty S:=\sup\bigl\{|\cS-\cS'|_\infty:\cS,\cS'\in S\bigr\}\in[0,+\infty], 
\end{gather*}
so in particular that for every $\cS\in\sA^1$, $\diam\cS=\diam_\infty\{\cS\}$.

For every $\cS_0\in\sH^1$ and every compact subset $E$ in $\sA^1$,
noting that
$\log[\,\cdot\,,\cS_0]_g=-\rho(\cS_g,\,\cdot\wedge_{\cS_g}\cS_0)\ge-\rho(\cS_g,\cS_0)=\log[\cS_0,\cS_0]_g$ on $\sH^1$ (by \eqref{eq:Gromov}), we have the uniform comparison
\begin{gather}
 \bigl([\cS_0,\cS_0]_g\bigr)^2[\cS,\cS']_{\cS_0}\le |\cS-\cS'|_\infty
\le\frac{[\cS,\cS']_{\cS_0}}{(\inf_E[\cdot,\infty]_g)^2}\quad\text{on }E\times E\label{eq:unifcompare}
\end{gather}
between the kernel functions $|\cS-\cS'|_\infty$ and $[\cS,\cS']_{\cS_0}$ restricted 
to $E\times E$. 

For every $\cS_0\in\sA^1$ and every $r\ge\diam(\cS_0)$, 
writing $B_{\cS_0}=B(z_0,\diam(\cS_0))$ for some $z_0\in K$, set
\begin{gather*}
 \sB(\cS_0,r):=\{\cS\in\sA^1:|\cS-\cS_0|_\infty\le r\}
=\sP^1\setminus U_{\overrightarrow{\cS_{B(z_0,r)}\infty}}\bigl(=\overline{B(z_0,r)}\quad\text{in }\sP^1\bigr).
\end{gather*}
Similarly, for every $a\in\bP^1$ and $r\in[0,1]$, set
\begin{gather*}
 \sB_\#(a,r):=\{\cS\in\sP^1:[\cS,a]_g\le r\}\bigl(=\overline{\{z\in\bP^1:[z,a]\le r\}}\quad\text{in }\sP^1\bigr),
\end{gather*}
which is called a Berkovich chordal closed ball in $\sP^1$.

\subsection{Classical Julia/Fatou sets and Hsia's equicontinuity criterion}\label{sec:classical}
For a rational function $f\in K(z)$ on $\bP^1$
of degree $>1$, the classical Fatou set of $f$
is defined by the set of all points in $\bP^1$ 
at each of which 
the iterations family $f^n:(\bP^1,[z,w])\to(\bP^1,[z,w])$, $n\in\bN$,
of $f$ is equicontinuous (then this subset is indeed open from
the proof of \cite[Theorem 23]{KS09}), 
and the classical Julia set is by the complement of the classical Fatou set of $f$ in $\bP^1$. 
Then in fact $\sJ(f)\cap\bP^1$ and $\sF(f)\cap\bP^1$ coincide
with the classical Julia and Fatou sets of $f$, respectively,
and $\sJ(f)\cap\bP^1$ still has no isolated points in $\bP^1$
(otherwise $\sJ(f)\cap\bP^1$ must be
a singleton in $\bP^1$, and in turn
$\sJ(f)\cap\bP^1\subset E(f)$, which contradicts $E(f)\subset\sF(f)$).

Hsia's non-archimedean counterpart \cite{Hsia00} to Montel's so called 
three points theorem asserts that, {\em if a family $\cF$ of 
meromorphic functions on a $K$-closed disk $B=B(a,r)$, $r>0$,
omits more than one point in $\bP^1$ in that $\#(\bP^1\setminus\bigcup_{f\in\cF}f(B))\ge 2$, then $\cF$ is 
equicontinuous on $B$
regarding each element of $\cF$ as a continuous 
mapping from $B$ to $(\bP^1,[z,w])$.} Let us call this theorem
Hsia's Montel-type {\em two points} theorem. For further developments
of the Montel-type equicontinuity/normality theorem in non-archimedean and Berkovich (i.e., non-classical) setting, 
see \cite{FKT12,Rodriguez16}.

\section{Proof of Theorem \ref{th:upJulia}}
\label{sec:upJulia}

Let $f\in K(z)$ be a rational function on $\bP^1$ of degree $d>1$
having no potentially good reductions (otherwise $\sJ(f)$ is a singleton,
so is not (uniformly) perfect).
The following argument is an adaption of that
in \cite{MR92}, which seems most adaptable to non-archimedean setting
among the (independent) originals \cite{MR92,Hinkkanen93,EremenkoUP}.

Suppose to the contrary that $\sJ(f)$ is not uniformly perfect, i.e., that
there is a sequence $(A_j)$ of (non-degenerating and finite open Berkovich)
annuli in $\sP^1$ separating $\sJ(f)$ such that 
$\lim_{j\to\infty}\operatorname{mod}(A_j)=+\infty$. 
Then without loss of generality, we assume that for every $j\in\bN$,
$A_j=A(\cS_j,\cS_j')$, where 
\begin{gather*}
 \cS_j=\cS_{B(a_j,|z_j|)}\quad\text{and}\quad\cS'_j=\cS_{B(a_j,|z'_j|)} 
\end{gather*}
for some $a_j\in\cO_K=B(0,1)$ and some $z_j,z'_j\in K^*$, $|z_j|<|z_j'|\le 1$
(so in particular $\lim_{j\to\infty}|z_j|=0$); 
indeed, (i) by \eqref{eq:modulusdecomp},
we assume that $\cS_g\not\in A_j$ for every $j\in\bN$, 
and in turn, (ii) by \eqref{eq:modulusdecomp},
the $\PGL(2,K)$-invariance of the moduli of annuli in $\sP^1$, 
and the $\PGL(2,K)$-equivariance of $\sJ(f)$, 
taking a subsequence of $(A_j)$ and 
replacing $f(z)$ with $1/f(1/z)$, if necessary,
we also assume that for every $j\in\bN$, 
$A_j=A(\cS_j,\cS_j')$ for some $\cS_j,\cS_j'\in\sH^1$
satisfying $\cS_j\prec\cS_j'\preceq\cS_g$, and finally, (iii) 
by the density of $\sH^1_{\mathrm{II}}$ in $\sP^1$ and
the continuity of $\diam$ on any interval in $\sP^1$, we assume that
for every $j\in\bN$, $\cS_j,\cS_j'\in\sH^1_{\mathrm{II}}$,
without loss of generality.

Pick sequences $(w_j),(w'_j)$ in $K^*$ such that for every $j\in\bN$,
\begin{gather*}
 \exp\bigl(-\operatorname{mod}(A_j)\bigr)=|z_j/z'_j|<|w_j|<|w'_j|< 1
\end{gather*}
and that $\lim_{j\to\infty}|w'_j|=\lim_{j\to\infty}|w_j|/|w'_j|=0$
(so $\lim_{j\to\infty}|w_j|=0$). 
For every $j\in\bN$, 
since $|z_j|<|z'_jw_j|<|z'_jw'_j|<|z'_j|$, we have
\begin{gather}
 \emptyset\neq\bigl(\sJ(f)\cap(\sP^1\setminus U_{\overrightarrow{\cS_j\cS'_j}})\subset\bigr)
\sJ(f)\cap U_{\overrightarrow{\cS_{B(a_j,|z'_jw_j|)}a_j}},
\label{eq:separate} 
\end{gather}
and then (see Benedetto \cite[Proposition 12.9]{BenedettoBook})
\begin{gather}
\sJ(f)\subset f^n(U_{\overrightarrow{\cS_{B(a_j,|z'_jw_j|)}a_j}})
\quad\text{for }n\gg 1.\label{eq:selfsimilar} 
\end{gather}

Let us also pick a subset $S\subset\bP^1$ such that $\#S\ge 3$ as follows;
since $\sup_{j\in\bN}|z'_j|\le 1$ and $\lim_{j\to\infty}|w_j|=0$, we also have
\begin{gather}
\diam_\infty(U_{\cS_{\overrightarrow{B(a_j,|z'_jw_j|)a_j}}})=
\diam(B(a_j,|z'_jw_j|))=|z'_j||w_j|\to 0\tag{\dag}
\label{eq:diameter}
\end{gather}
as $j\to\infty$. Hence if {\bfseries (a)} the sequence $(a_j)$ (in $\cO_K$) has 
a limit point in $\cO_K$, then by \eqref{eq:separate},
\eqref{eq:diameter}, and the strong triangle inequality, this limit point is
in $\sJ(f)\cap\bP^1$, and then since $\sJ(f)\cap\bP^1$ has no isolated points, 
we can pick $S\subset\sJ(f)\cap\bP^1$ such that $\#S=3$.
Alternatively if {\bfseries (b)} the sequence
$(a_j)$ (in $\cO_K$) has no limit points in $\cO_K$, 
then by \eqref{eq:diameter} and the strong triangle inequality
(and the completeness of $K$),
we can fix $1\ll j_1<j_2<j_3$ such that
$U_{\overrightarrow{\cS_{B(a_{j_1},|z'_{j_1}w_{j_1}|)}a_{j_1}}}\cap U_{\overrightarrow{\cS_{B(a_{j_2},|z'_{j_2}w_{j_2}|)}a_{j_2}}}=\emptyset$ and in turn that
$(U_{\overrightarrow{\cS_{B(a_{j_1},|z'_{j_1}w_{j_1}|)}a_{j_1}}}\cup U_{\overrightarrow{\cS_{B(a_{j_2},|z'_{j_2}w_{j_2}|)}a_{j_2}}})
\cap U_{\overrightarrow{\cS_{B(a_{j_3},|z'_{j_3}w_{j_3}|)}a_{j_3}}}=\emptyset$,
and set $S:=\{a_{j_\ell}:\ell\in\{1,2,3\}\}$.

Fix $L>1$ so large that $(0<)L^{-1}\cdot\min_{z,w\in S:\text{distinct}}[z,w]\le\diam_\#(\sJ(f)\cap\bP^1)$ and that 
the action on $(\bP^1,[z,w])$ of $f$ is $L$-Lipschitz continuous,
and fix $c\in(0,L^{-1}\cdot\min_{z,w\in S:\text{distinct}}[z,w])$. Then
by \eqref{eq:selfsimilar},
there exists a (unique) sequence $(m_j)$ in $\bN$ tending to $\infty$ 
as $j\to\infty$ such that for every $j\gg 1$, 
\begin{multline*}
 \max_{n\in\{0,1,\ldots,m_j-1\}}\diam_\#\bigl(f^n(U_{\overrightarrow{\cS_{B(a_j,|z'_jw_j|)}a_j}})\cap\bP^1\bigr)<c\\
\le\diam_\#\bigl(f^{m_j}(U_{\overrightarrow{\cS_{B(a_j,|z'_jw_j|)}a_j}})\cap\bP^1\bigr)\bigl(<Lc<\min_{z,w\in S:\text{distinct}}[z,w]\le 1\bigr), 
\end{multline*}
so in particular
\begin{gather}
 \#\bigl((f^{m_j}(U_{\overrightarrow{\cS_{B(a_j,|z'_jw_j|)}a_j}}))\cap S\bigr)\le 1.\tag{*}\label{eq:atmostone}
\end{gather}
Moreover, by the choice of $(w_j),(w'_j)$, for every $j\in\bN$, 
we have $U_{\overrightarrow{\cS_{B(a_j,|z_j'w'_j|)}a_j}}
\setminus U_{\overrightarrow{\cS_{B(a_j,|z_j'w_j|)}a_j}}\subset A_j$, 
which with $f(\sF(f))\subset\sF(f)$ also yields
\begin{gather}
 f^{m_j}\bigl(U_{\overrightarrow{\cS_{B(a_j,|z_j'w'_j|)}a_j}}
\setminus U_{\overrightarrow{\cS_{B(a_j,|z_j'w_j|)}a_j}}\bigr)\subset
f^{m_j}(A_j)\subset\sF(f).\tag{**}\label{eq:annulusFatou}
\end{gather} 

We claim that for $j\gg 1$, we still have
\begin{gather*}
 \#\bigl((f^{m_j}(U_{\overrightarrow{\cS_{B(a_j,|z'_jw'_j|)}a_j}}))\cap S\bigr)\le 1; 
\end{gather*}
for, in the case (a), the claim holds 
by \eqref{eq:atmostone} and \eqref{eq:annulusFatou}
since $S\subset\sJ(f)$ in this case. 
In the case (b), for $j\gg 1$, 
by \eqref{eq:image} and \eqref{eq:atmostone}, we have 
\begin{gather}
 f^{m_j}(U_{\overrightarrow{\cS_{B(a_j,|z'_jw_j|)}a_j}})
=U_{(f^{m_j})_*(\overrightarrow{\cS_{B(a_j,|z'_jw_j|)}a_j})}\tag{***}.\label{eq:imagedirect}
\end{gather}
Then 
$f^{m_j}(U_{\overrightarrow{\cS_{B(a_j,|z'_jw_j|)}a_j}})$
intersects $U_{\overrightarrow{\cS_{B(a_{j_\ell},|z'_{j_\ell}w_{j_\ell}|)}\cS_{j_\ell}}}$ for at most one of $\ell\in\{1,2,3\}$ (by the strong triangle inequality), 
and then by \eqref{eq:image}, \eqref{eq:separate}, and \eqref{eq:annulusFatou},
we even have 
\begin{gather}
 f^{m_j}(U_{\overrightarrow{\cS_{B(a_j,|z_j'w'_j|)}a_j}})
=U_{(f^{m_j})_*(\overrightarrow{\cS_{B(a_j,|z_j'w'_j|)}a_j})}.\tag{****}\label{eq:imagedirectlarge}
\end{gather}
Hence, if $\#\bigl\{\ell\in\{1,2,3\}:U_{(f^{m_j})_*(\overrightarrow{\cS_{B(a_j,|z_j'w'_j|)}a_j})}\cap U_{\overrightarrow{\cS_{B(a_{j_\ell},|z'_{j_\ell}w_{j_\ell}|)}a_{j_\ell}}}\neq\emptyset\bigr\}\le 1$, then we are done. Otherwise,
$U_{(f^{m_j})_*(\overrightarrow{\cS_{B(a_j,|z_j'w'_j|)}a_j})}\supset
U_{\overrightarrow{\cS_{B(a_{j_\ell},|z'_{j_\ell}w_{j_\ell}|)}a_{j_\ell}}}$
for more than one of $\ell\in\{1,2,3\}$
(by the strong triangle inequality), and then
\begin{multline*}
\emptyset=\sF(f)\cap\sJ(f)
\underset{\text{by }\eqref{eq:annulusFatou}}{\supset}\bigl(f^{m_j}(U_{\overrightarrow{\cS_{B(a_j,|z_j'w'_j|)}a_j}}
\setminus U_{\overrightarrow{\cS_{B(a_j,|z_j'w_j|)}a_j}})\bigr)\cap\sJ(f)\\
\underset{\text{by }\eqref{eq:imagedirect}\&\eqref{eq:imagedirectlarge}}{\supset}
\bigl(U_{(f^{m_j})_*(\overrightarrow{\cS_{B(a_j,|z_j'w'_j|)}a_j})}\setminus U_{(f^{m_j})_*(\overrightarrow{\cS_{B(a_j,|z'_jw_j|)}a_j})}\bigr)\cap\sJ(f)\\
\underset{\text{by }\eqref{eq:atmostone}\&\eqref{eq:separate}}{\supset} 
(U_{\overrightarrow{\cS_{B(a_{j_\ell},|z'_{j_\ell}w_{j_\ell}|)}a_{j_\ell}}})\cap\sJ(f)\neq\emptyset
\end{multline*}
for at least one of $\ell\in\{1,2,3\}$ (the final inclusion in which is also by the strong triangle inequality). This is impossible, and we are also done.

Once the claim is at our disposal, taking a subsequence if necessary, 
there is $S'\subset S$ such that 
$\#S'=2$ and that the family 
\begin{gather*}
 g_j(z):=f^{m_j}\bigl(a_j+z'_jw'_jz\bigr)\in K(z), \quad j\in\bN, 
\end{gather*}
restricted to $\cM_K=U_{\overrightarrow{\cS_g0}}\cap K$ 
omits $S'$. Then by Hsia's Montel-type two points theorem 
(see \S\ref{sec:classical}), 
this family $\{g_j:\cM_K\to(\bP^1,[z,w])\}_{j\in\bN}$ is equicontinuous 
at $z=0$. Consequently, recalling that
$\lim_{j\to\infty}|w_j/w'_j|=0$ 
(from the choice of $(w_j),(w'_j)$), we must have
\begin{multline*}
 0=\lim_{j\to\infty}\diam_\#\bigl(g_j(B(0,|w_j/w'_j|))\bigr)\\
=\liminf_{j\to\infty}\diam_\#\bigl(f^{m_j}(U_{\overrightarrow{\cS_{B(a_j,|z'_jw_j|)}a_j}})\cap\bP^1\bigr)\ge c>0,
\end{multline*}
which is a contradiction. Now the proof of Theorem \ref{th:upJulia} is complete. \qed

\section{A (non-infinitesimal) lower capacity density property}
\label{sec:lcd}

Let $K$ be an algebraically closed field that is complete
with respect to a non-trivial and non-archimedean absolute value $|\cdot|$.

The logarithmic capacity $\Capa_\infty E$ with pole $\infty$
of a compact subset $E$ in $\sA^1$ 
(see \eqref{eq:capacitygeneral} for the definition)
equals the transfinite diameter 
\begin{gather*}
d^{(\infty)}_\infty E:=\lim_{n\to\infty}\Bigl(\sup_{\cS_1,\ldots,\cS_n\in E}\Bigl(\prod_{i=1}^n\prod_{j\in\{1,\ldots,n\}\setminus\{i\}}\bigl|\cS_i-\cS_j\bigr|_\infty\Bigr)^{\frac{1}{n(n-1)}}\Bigr)\in\bR_{\ge 0}
\end{gather*}
of $E$ with pole $\infty$ (the above limit is a decreasing one); in particular,
$\Capa_\infty \emptyset=0$ and $\Capa_\infty E\le\Capa_\infty(E')$ 
if $E\subset E'$. 
For every $\cS\in\sA^1$ and every $r\ge\diam\cS$,
writing $B_{\cS}=B(z_0,\diam\cS)$, we have
\begin{gather*}
 \Capa_\infty(\sB(\cS,r))=\Capa_\infty(\{\cS_{B(z_0,r)}\})=r.
\end{gather*}
More details on potential theory on $\sP^1$ would be
given in Section \ref{sec:holder}.

\begin{definition}
We say a compact subset $E$ in $\sP^1$ 
has the (non-infinitesimal) lower capacity density property
(for some $c\in(0,1)$) if 
there is $c\in(0,1)$ such that for every $\cS\in E\setminus\{\infty\}$ 
and every $r\in(\diam\cS,\diam_\infty E)$, 
\begin{gather}
  \Capa_\infty\bigl(E\cap\sB(\cS,r)\bigr)\ge cr.\label{eq:lcd} 
\end{gather}
\end{definition}

The following potential theoretic precision/characterization
of compact subsets in $\sP^1$
having the bounded moduli property for separating annuli
(see \eqref{eq:boundedmoduli})
is a non-archimedean counterpart to the fundamental Pommerenke \cite[Theorem 1]{Pom79}
in complex analysis.

\begin{mainth}\label{th:lcd}
Let $K$ be an algebraically closed field that is complete
with respect to a non-trivial and non-archimedean absolute value, and
let $E$ be a compact subset in $\sP^1$. 
Then $E$ has the bounded moduli property for its separating annuli
if and only if $E$ has the $($non-infinitesimal$)$
lower capacity density property $($for some $c\in(0,1))$.
\end{mainth}

Among several proofs of \cite[Theorem 1]{Pom79} 
(see the survey \cite{sugawa03}),
the beautiful argument in Pommerenke's original one seems 
most adaptable to non-archimedean setting.

\begin{proof}[Proof of Theorem $\ref{th:lcd}$]
Let $E$ be a compact subset in $\sP^1$, and set
\begin{gather*}
c_E:=\sup\bigl\{\operatorname{mod}A:A\text{ is a concentric annulus in }\sA^1
 \text{ separating }E\bigr\}.
\end{gather*}
\subsection*{[$\Leftarrow$]}
For any (non-degenerating and finite Berkovich open) concentric annulus 
$A(\cS,\cS')$ (for some $\cS,\cS'\in\sH^1$ satisfying $\cS\prec\cS'$) 
in $\sA^1$ separating $E$ and
any $\diam\cS<r<r'<\diam(\cS')$ (so $r'\le\diam_\infty E$), we have
\begin{gather*}
  \Capa_\infty(E\cap\sB(\cS,r'))=\Capa_\infty(E\cap\sB(\cS,r))\le\Capa_\infty(\sB(\cS,r))=r. 
\end{gather*}
Hence if $E$ has
the (non-infinitesimal) lower capacity density property (for some $c\in(0,1)$), 
then $cr'\le r$, i.e., $\log(r'/r)\le\log(1/c)$.
Making $r\searrow\diam\cS$ and $r'\nearrow\diam(\cS')$,
we have $\operatorname{mod}(A(\cS,\cS'))\le\log(1/c)$, so that
$c_E\le\log(1/c)<+\infty$. 

\subsection*{[$\Rightarrow$]} 
Suppose $c_E\in\bR_{\ge 0}$. Fix
$s\in(0,e^{-c_E})(\subset(0,1))$, and pick $\cS_0\in E$ 
and $r\in(\diam(\cS_0),\diam_\infty E)$.

To any $j\in\bN\cup\{0\}$ and any $\cS\in E\setminus\{\infty\}$,
we can associate a point $[\cS]_j=[\cS]_j\in E\setminus\{\infty\}$ 
satisfying
\begin{gather*}
\begin{cases}
[\cS]_j=\cS & \text{if }\diam\cS\ge e^{-c_E}s^jr,\\
[\cS]_j\in E\cap\bigl\{\cS'\in\sA^1:e^{-c_E}s^jr
\le|\cS'-\cS|_\infty\le s^jr\bigr\} & \text{if }\diam\cS< e^{-c_E}s^jr;
\end{cases}
\end{gather*}
in the latter case,
if $\cS\in\sH^1_{\mathrm{II}}\cup\sH^1_{\mathrm{III}}$, then
writing $\cS=B(a,\diam\cS)$ for some $a\in K$, $\{\cS'\in\sA^1:e^{-c_E}s^jr
\le |\cS'-\cS|_\infty\le s^jr\}=A(\cS_{B(a,e^{-c_E}s^jr)},\cS_{B(a,s^jr)})$.
The case that $\cS\in\sH^1_{\mathrm{IV}}$, i.e.,
that $\cS$ is the cofinal equivalence class of a non-increasing and nesting
sequence $(B(a_i,r_i))$ of $K$-closed disks having the empty intersection, 
is treated in a similar way.

Then inductively on $j\in\bN$, 
for every $j\in\bN$ and every $(i_1,\ldots,i_j)\in\{0,1\}^j$,
we set $a_{i_1,\ldots,i_j,i_{j+1}}\in E\setminus\{\infty\}$ 
so that
\begin{gather*}
a_{i_1}=\begin{cases}
	  \cS_0 &\text{if }i_1=0,\\
	  [\cS_0]_0 & \text{if }i_1=1
	 \end{cases}
\quad\text{and}\quad
a_{i_1,\ldots,i_j,i_{j+1}}=\begin{cases}
		      a_{i_1,\ldots,i_j} &\text{if }i_{j+1}=0,\\
		      [a_{i_1,\ldots,i_j}]_j &\text{if }i_{j+1}=1,
		     \end{cases}
\end{gather*}
and for every $j\in\bN$, set
\begin{gather*}
 E_j:=\bigl\{a_{i_1,\ldots,i_j}:(i_1,\ldots,i_j)\in\{0,1\}^j\bigr\},
\end{gather*}
so that $\cS_0\in E_j\subset E_{j+1}$. 
For every $j\in\bN$, using the strong triangle inequality
for the Hsia kernel $|\cS-\cS'|_\infty$ repeatedly, we have 
\begin{gather*}
 \sup_{(i_1,\ldots,i_j)\in\{0,1\}^j}|a_{i_1,\ldots,i_j}-\cS_0|_\infty
\le\max\{\diam(\cS_0),s^0r\}=r, 
\end{gather*}
so that 
$E_j\subset E\cap\sB(\cS_0,r)$.

For every $j\in\bN$ and any distinct $(i_1,\ldots,i_j),(i'_1,\ldots,i'_j)\in\{0,1\}^j$, writing 
\begin{gather*}
 \cS=a_{i_1,\ldots,i_j}\quad\text{and}\quad\cS'=a_{i'_1,\ldots,i'_j}
\end{gather*}
(in $E_j$) for simplicity, set
\begin{gather*}
m_*=m_*(\cS,\cS'):=\max\bigl\{m\in\{1,\ldots,j\}:i_m=i'_m\bigr\}\in\{0,1,\ldots,j-1\}
\end{gather*}
(under the convention $\max\emptyset=0$) and
\begin{gather*}
 \cS_*=\cS_*(\cS,\cS'):=
\begin{cases}
 a_{i_1,\ldots,i_{m_*}}=a_{i_1',\ldots,i_{m_*}'}
&\text{if }m_*>0,\\
 \cS_0 &\text{if }m_*=0
\end{cases}\in E_j. 
\end{gather*}

We claim that 
\begin{gather*}
 |\cS-\cS'|_\infty>s^{m_*+1}r;
\end{gather*}
indeed, replacing $(i_1,\ldots,i_j),(i'_1,\ldots,i'_j)$ if necessary,
we have $i_{m_*+1}=0$ and $i'_{m_*+1}=1$, so that 
\begin{gather*}
 a_{i_1,\ldots,i_{m_*},i_{m_*+1}}=\cS_*
\quad\text{and}\quad a_{i_1',\ldots,i_{m_*},i_{m_*+1}'}=[\cS_*]_{m_*}.
\end{gather*}
If $\max\{\diam(\cS_*),\diam([\cS_*]_{m_*})\}>s^{m_*+1}r$, then
either $\cS=\cS_*$ or $\cS'=[\cS_*]_{m_*}$, and in turn,
also using the lower estimate \eqref{eq:Hsiadiam}
of the kernel function $|\cS-\cS'|_\infty$, we have the desired estimate
\begin{gather*}
 |\cS-\cS'|_\infty\ge\max\bigl\{\diam(\cS_*),\diam([\cS_*]_{m_*})\bigr\}>s^{m_*+1}r
\end{gather*}
in this case. Alternatively 
if $\max\{\diam(\cS_*),\diam([\cS_*]_{m_*})\}\le s^{m_*+1}r$, then
\begin{gather*}
 \max_{m\in\{m_*+1,\cdots,j-1\}}
\max\bigl\{\diam(a_{i_1,\ldots,i_m}),\diam(a_{i'_1,\ldots,i'_m})\bigr\}
\le s^{m_*+1}r, 
\end{gather*}
so that
also using the strong triangle inequality for the kernel function $|\cS-\cS'|_\infty$
repeatedly,
we compute
\begin{multline*}
 |\cS-\cS_*|_\infty
\le\max_{m\in\{m_*+1,\ldots,j-1\}}\bigl|a_{i_1,\ldots,i_m,i_{m+1}}-a_{i_1,\ldots,i_m}\bigr|_\infty\\
\le\max_{m\in\{m_*+1,\ldots,j-1\}}\max\bigl\{\diam(a_{i_1,\ldots,i_m}),s^mr\bigr\}
\le s^{m_*+1}r
\end{multline*}
and, similarly,
\begin{gather*}
\bigl|\cS'-[\cS_*]_{m_*}\bigr|_\infty
\le\max_{m\in\{m_*+1,\ldots,j-1\}}\bigl|a_{i'_1,\ldots,i'_m,i'_{m+1}}-a_{i'_1,\ldots,i'_m}\bigr|_\infty\le s^{m_*+1}r.
\end{gather*}
By the above two estimates 
and the strong triangle inequality for the kernel function $|\cS-\cS'|_\infty$ again,
we have
\begin{multline*}
\max\bigl\{|\cS-\cS'|_\infty,s^{m_*+1}r\bigr\}\\
\ge\max\bigl\{|\cS-\cS_*|_\infty,|\cS_*-[\cS_*]_{m_*}|_\infty,|\cS'-[\cS_*]_{m_*}|_\infty\bigr\}\\
\ge \bigl|\cS_*-[\cS_*]_{m_*}\bigr|_\infty\ge e^{-c_E}s^{m_*}r>s^{m_*+1}r,
\end{multline*}
which also yields the desired estimate in this case. Hence the claim holds.

Once the claim is at our disposal, for 
every $j\in\bN$ and every $(i_1,\ldots,i_j)\in\{0,1\}^j$,
we compute
\begin{align*}
 &\prod_{(i'_1,\ldots,i'_j)\in\{0,1\}^j\setminus\{(i_1,\ldots,i_j)\}}|a_{i_1,\ldots,i_j}-a_{i'_1,\ldots,i'_j}|_\infty\\
=&\prod_{m=0}^{j-1}\Bigl(\prod_{(i'_1,\ldots,i'_j)\in\{0,1\}^j\setminus\{(i_1,\ldots,i_j)\}:m_*(a_{i_1,\ldots,i_j},a_{i'_1,\ldots,i'_j})=m}|a_{i_1,\ldots,i_j}-a_{i'_1,\ldots,i'_j}|_\infty\Bigr)\\
>&\prod_{m=0}^{j-1}(s^{m+1}r)^{2^{(j-m)-1}}
=s^{\sum_{m=0}^{j-1}(m+1)2^{(j-m)-1}}\cdot r^{2^j-1},
\end{align*}
so that, also recalling that $E_j\subset E\cap\sB(\cS_0,r)$ 
and $\#\{0,1\}^j=2^j$, we have
\begin{multline*}
 \Capa_\infty\bigl(E\cap\sB(\cS_0,r)\bigr)=d^{(\infty)}_\infty\bigl(E\cap\sB(\cS_0,r)\bigr)\\
\ge\limsup_{j\to\infty}
\Bigl(\bigl(s^{2^j\sum_{m=0}^{j-1}(m+1)2^{-m-1}}\cdot\, r^{2^j-1}\bigr)^{2^j}\Bigr)^{\frac{1}{2^j(2^j-1)}}\\
=\bigl(s^{\sum_{m=0}^\infty\frac{m+1}{2^{m+1}}}\bigr)r=s^2r.
\end{multline*}
Hence $E$ has the $($non-infinitesimal$)$
lower capacity density property (for $c=e^{-2c_E}$).
\end{proof}

\section{A uniform H\"older continuity property}
\label{sec:holder}

Let $K$ be an algebraically closed field that is complete
with respect to a non-trivial and non-archimedean absolute value $|\cdot|$.

Let us recall some details on logarithmic potential theory on $\sP^1$
(see \cite[\S 6]{BR10}). For every $\cS_0\in\sP^1$ and 
every positive Radon measure $\nu$ on $\sP^1$,
the logarithmic potential of $\nu$ on $\sP^1$ with 
respect to  $\cS_0$ (or with pole $\cS_0$ when $\cS_0\in\bP^1$) is the function
\begin{gather*}
 p_{\cS_0,\nu}(\cdot):=\int_{\sP^1}\log[\cdot,\cS']_{\cS_0}\nu(\cS'):\sP^1\to[-\infty,+\infty].
\end{gather*}
If either $\cS_0\in\sH^1$ or $\supp\nu\subset\sP^1\setminus\{\cS_0\}$, then 
$p_{\cS_0,\nu}$ is 
strongly upper semicontinuous on $\sP^1$ in that for every $\cS\in\sP^1$,
$\limsup_{\cS'\to\cS}p_{\cS_0,\nu}(\cS')=p_{\cS_0,\nu}(\cS)$, and satisfies
\begin{gather*}
p_{\cS_0,\nu}(\cS_0)=-\log[\cS_0,\cS_0]_g=\max_{\sP^1}p_{\cS_0,\nu}. 
\end{gather*}
If $\supp\nu\subset\sP^1\setminus\{\cS_0\}$, then 
$p_{\cS_0,\nu}$ is 
not only subharmonic on $\sP^1\setminus\{\cS_0\}$
but also harmonic on $\sP^1\setminus(\{\cS_0\}\cup\supp\nu)$,
and the logarithmic energy of $\nu$ with respect to $\cS_0$ 
(or with pole $\cS_0$ when $\cS_0\in\bP^1$) is defined as 
\begin{gather*}
 I_{\cS_0,\nu}:=\int_{\sP^1}p_{\cS_0,\nu}\nu=\int_{\sP^1\times\sP^1}
\log[\cS,\cS']_{\cS_0}(\nu\times\nu)(\cS,\cS')\in[-\infty,+\infty).
\end{gather*}
For every subset $C$ in $\sP^1$ and every $\cS_0\in\bP^1\setminus C$,
setting 
\begin{gather*}
 V_{\cS_0}(C):=\sup_{\nu}I_{\cS_0,\nu}\in[-\infty,+\infty),
\end{gather*}
where $\nu$ ranges over all probability Radon measures on $\sP^1$
supported by $C$ (under the convention $\sup_{\emptyset}=-\infty$ here),
the logarithmic capacity of $C$ with respect to $\cS_0$ 
(or with pole $\cS_0$ when $\cS_0\in\bP^1$) is 
\begin{gather}
 \Capa_{\cS_0}(C):=\exp\bigl(V_{\cS_0}(C)\bigr)\in\bR_{\ge 0}.\label{eq:capacitygeneral}
\end{gather}
For every compact subset $E$ in $\sP^1$
and every $\cS_0\in\bP^1\setminus E$, if $\Capa_{\cS_0}(E)>0$, then 
the following Frostman-type properties hold;
there is a unique probability Radon measure $\nu$ on $\sP^1$, 
which is called the equilibrium mass distribution on $E$
with respect to $\cS_0$ (or with pole $\cS_0$ when $\cS_0\in\bP^1\setminus E$)
and is denoted by $\nu_{\cS_0,E}$, 
such that $\supp\nu\subset E$ and that $I_{\cS_0,\nu}=V_{\cS_0}(E)$;
then denoting by $D_{\cS_0}=D_{\cS_0,E}$ the component of
$\sP^1\setminus E$ containing $\cS_0$, we have
$\supp(\nu_{\cS_0,E})\subset\partial(D_{\cS_0})$,
\begin{gather*}
p_{\cS_0,\nu_{\cS_0,E}}\ge V_{\cS_0}(E)\text{ on }\sP^1,\quad
p_{\cS_0,\nu_{\cS_0,E}}>V_{\cS_0}(E)\text{  on }D_{\cS_0},\quad\text{and}\\
p_{\cS_0,\nu_{\cS_0,E}}\equiv V_{\cS_0}(E)\text{ on }\sP^1\setminus(D_{\cS_0}\cup F),
\end{gather*}
and $p_{\cS_0,\nu_{\cS_0,E}}$ is continuous at every point in 
$\sP^1\setminus F$, 
where $F$ is an $F_{\sigma}$-subset in $(\partial(D_{\cS_0}))\cap\bP^1$ satisfying $\Capa_{\cS_0}(F)=0$;
$F$ could be chosen as $\emptyset$ if and only if
$p_{\cS_0,\nu_{\cS_0,E}}$ is continuous at any point in $\partial(D_{\cS_0})$,
and then $\supp(\nu_{\cS_0,E})=\partial(D_{\cS_0})$.
For the details on subharmonic 
functions on open subsets in $\sP^1$, see \cite[\S7, \S8]{BR10} and
\cite[\S 3]{ThuillierThesis}. 

By the $\PGL(2,\cO_K)$-invariance \eqref{eq:Hsiainvariant}
of the generalized Hsia kernel functions $[\cS,\cS']_{\cS_0}$ on $\sP^1$,
both the logarithmic capacity $\Capa_{\cS_0}(E)$ of a compact subset $E$ in $\sP^1$
and the equilibrium mass distribution on $E$ in $\sP^1$ 
with respect to $\cS_0\in\sP^1\setminus E$
are $\PGL(2,\cO_K)$-invariant/equivariant
in that for every $\cS_0\in\sP^1\setminus E$
and every $M\in\PGL(2,\cO_K)$, 
\begin{gather*}
 \Capa_{M(\cS_0)}(M(E))=\Capa_{\cS_0}E
\end{gather*}
and, if in addition $\Capa_{\cS_0}E>0$, then $\nu_{M(\cS_0),M(E)}=M_*(\nu_{\cS_0,E})$
on $\sP^1$.

In computation, the so called Perron method is useful
(see \cite[\S 3.1.3]{ThuillierThesis}).
For every $\cS_0\in\sH^1$, every $\overrightarrow{v}\in T_{\cS_0}\sP^1$, 
and every compact subset $E$ in $U_{\overrightarrow{v}}$, the 
zero-one extremal function 
\begin{multline*}
\cS\mapsto h(\cS;E,U_{\overrightarrow{v}})
:=\sup\bigl\{u(\cS):u\text{ is subharmonic on }U_{\overrightarrow{v}},\, u|E\le 0,\\
\text{and }u< 1\text{ on }U_{\overrightarrow{v}}\bigr\}:U_{\overrightarrow{v}}\to[-\infty,1)
\end{multline*} 
for $E$ with respect to $U_{\overrightarrow{v}}$
is not only harmonic on $U_{\overrightarrow{v}}\setminus E$
but also $\equiv 0$ on the union of $E$ and all components 
of $(U_{\overrightarrow{v}}\cup\{\cS_0\})\setminus E$ 
not containing $\cS_0$, and
the upper semicontinuous regularization $h(\cdot;E,U_{\overrightarrow{v}})^*$
of $h(\cdot;E,U_{\overrightarrow{v}})$ on $U_{\overrightarrow{v}}$ 
is subharmonic on $U_{\overrightarrow{v}}$ and
coincides with $h(\cdot;E,U_{\overrightarrow{v}})$
on $U_{\overrightarrow{v}}\setminus E$. 
On the other hand, similarly, for every compact subset $E$ in $\sP^1$
and every $\cS_0\in\sH^1\setminus E$, 
if $\Capa_{\cS_0}(E)>0$, then the function
\begin{multline*} 
\cS\mapsto P_{\cS_0,E}(\cS)
:=\sup\bigl\{u(\cS):u\text{ is subharmonic on }\sP^1\setminus\{\cS_0\},\, u|E\le 0,\\
\text{ and }u<-\log[\cS_0,\cS_0]_g-\log\Capa_{\cS_0}(E)\bigr\}\\
:\sP^1\to\bigl[-\infty,-\log[\cS_0,\cS_0]_g-\log\Capa_{\cS_0}(E)\bigr)
\end{multline*}
is superharmonic on the $D_{\cS_0}=D_{\cS_0,E}$, is harmonic on 
$D_{\cS_0}\setminus\{\cS_0\}$,
and is $\equiv 0$ on $\sP^1\setminus D_{\cS_0}$, 
and the upper semicontinuous regularization $P_{\cS_0,E}^*$ of $P_{\cS_0,E}$ 
on $\sP^1\setminus\{\cS_0\}$ is subharmonic on $\sP^1\setminus\{\cS_0\}$;
indeed,
\begin{gather}
P_{\cS_0,E}^*\equiv p_{\cS_0,\nu_{\cS_0,E}}-\log\Capa_{\cS_0}(E)\quad\text{on }\sP^1,\label{eq:perron}
\end{gather}
which is nothing but the generalized Green function $G_{\cS_0,E}$ on $\sP^1$ associated to $E$
with respect to $\cS_0$ (or with pole $\cS_0$ when $\cS_0\in\bP^1\setminus E$).

For every compact subset $E$ in $\sA^1$, 
every $\cS_0\in\sH^1\setminus E$, every $a\in E\cap K$, every $t>0$,
and every $R>1$ satisfying the following admissibility
\begin{gather}
 \Capa_{\cS_0}\bigl(E\cap\sB(a,t)\bigr)>0\quad\text{and}\quad
\cS_{B(a,Rt)}\in[\cS_0,a],
\label{eq:admissibility}
\end{gather}
we have not only
\begin{multline}
\frac{P_{\cS_0,E\cap\sB(a,t)}}{\sup_{U_{\overrightarrow{\cS_{B(a,Rt)}a}}}P_{\cS_0,E\cap\sB(a,t)}}\\\le 
h\bigl(\cdot;E\cap\sB(a,t),U_{\overrightarrow{\cS_{B(a,Rt)}a}}\bigr)\\
\le\frac{P_{\cS_0,E\cap\sB(a,t)}^*}{P_{\cS_0,E\cap\sB(a,t)}^*(\cS_{B(a,Rt)})}
\quad\text{on }U_{\overrightarrow{\cS_{B(a,Rt)}a}}
\label{eq:compare}
\end{multline} 
(cf.\ \cite[p.\ 209]{Siciak97})
using the maximum principle (\cite[Proposition 8.14]{BR10})
for subharmonic functions (on $U_{\overrightarrow{\cS_{B(a,Rt)}a}}$), but also
\begin{multline}
P_{\cS_0,E\cap\sB(a,t)}^*(\cS_{B(a,Rt)})
+\log\Capa_{\cS_0}\bigl(E\cap\sB(a,t)\bigr)\\
\Bigl(=p_{\cS_0,\nu_{\cS_0,E\cap\sB(a,t)}}(\cS_{B(a,Rt)})\Bigr)
=-\rho(\cS_0,\cS_{B(a,Rt)})+\rho(\cS_g,\cS_0)\label{eq:interm}
\end{multline}
using the identities \eqref{eq:perron} (applied to $E\cap\sB(a,t)$) and
$\log[\cS_{B(a,Rt)},\cdot\,]_{\cS_0}
\equiv-\rho(\cS_0,\cS_{B(a,Rt)})+\rho(\cS_g,\cS_0)$ 
on $\sB(a,t)$ (from \eqref{eq:Gromov}).
The latter \eqref{eq:interm} in particular yields
\begin{multline}
P_{\cS_0,E\cap\sB(a,t)}^*(\cS_{B(a,Rt)})\\
\Bigl(=-2\log\diam(\cS_0\wedge\cS_{B(a,Rt)})+\log\diam(\cS_0)+\log(Rt)\\
+2\log\diam(\cS_g\wedge\cS_0)-\log 1-\log\diam(\cS_0)-\log\Capa_{\cS_0}\bigl(E\cap\sB(a,t)\bigr)\Bigr)\\
\le 2\log\bigl(\diam(\cS_g\wedge\cS_0)/\diam(\cS_0)\bigr)
+\log\bigl(R/\bigl(\Capa_{\cS_0}(E\cap\sB(a,t))/t\bigr)\bigr).\label{eq:Greenlower}
\end{multline}

The following uniform H\"older continuity property 
of the generalized Green functions on $\sP^1$
associated to a non-empty compact subset 
in $\sP^1$ satisfying the (non-infinitesimal) lower capacity density condition
\eqref{eq:lcd}
is a non-archimedean counterpart to Lithner \cite{Lithner94} in complex analysis.
  
\begin{mainth}\label{th:holder}
Let $K$ be an algebraically closed field that is complete
with respect to a non-trivial and non-archimedean absolute value.
Then for every compact subset $E$ in 
$\sP^1$ having the $($non-infinitesimal$)$ 
lower capacity density property and every $\cS_0\in\sP^1\setminus E$,
denoting by $D_{\cS_0}=D_{\cS_0,E}$
the component of $\sP^1\setminus E$ containing $\cS_0$, 
there are $\alpha>0$ and $\delta_0>0$ such that 
\begin{gather*}
\sup_{a\in(\partial(D_{\cS_0}))\cap\bP^1}\sup_{\delta\in(0,\delta_0]}\frac{\sup_{\sB_\#(a,\delta)}\bigl(p_{\cS_0,\nu_{\cS_0,E}}(\cdot)-\log\Capa_{\cS_0}(E)\bigr)}{\delta^\alpha}<+\infty.
\end{gather*}
\end{mainth}

\begin{proof}
Pick a compact subset $E$ in $\sP^1$, 
which is for a while not assumed to have
the (non-infinitesimal) lower capacity density property, 
and pick a point $\cS_0\in\sP^1\setminus E$. 
Then by 
the above $\PGL(2,\cO_K)$-invariance/equivariance of
the kernel function $[\cS,\cS']_g$,
the (positive) logarithmic capacity, 
and the equilibrium mass distribution 
of a compact subset in $\sP^1$, we assume both
\begin{gather*}
 E\subset\sA^1
\end{gather*}
and $\cS_0\in D_\infty$, and moreover $\cS_0=\infty$ when $\cS_0\in\bP^1$, 
without loss of generality. We even assume 
\begin{gather*}
 \cS_0\in D_\infty\cap\sH^1 
\end{gather*}
without loss of generality; for, by the strong triangle inequality for 
the kernel function $[\cS,\cS']_g$,
if $\cS_0\in\sH^1$ is close enough to $\infty$, then
$[\cS,\cS']_{\cS_0}=[\cS,\cS']_\infty$ on $E\times E$. 
In the following, we adapt Siciak's Wiener-type argument in \cite{Siciak97}
to non-archimedean setting. 

\subsection*{(i)} For any $t>0$, any $R>r>1$, and any $a\in E\cap K$
satisfying the admissibility \eqref{eq:admissibility},
by the latter inequality in \eqref{eq:compare}, 
the maximum principle for subharmonic functions 
(on $U_{\overrightarrow{\cS_{B(a,rt)}a}}$), 
\eqref{eq:interm} twice, and \eqref{eq:Greenlower},
we compute
 \begin{multline}
c(t):=1-\sup_{U_{\overrightarrow{\cS_{B(a,rt)}a}}}h\bigl(\cdot;E\cap\sB(a,t),U_{\overrightarrow{\cS_{B(a,Rt)}a}}\bigr)\\
\biggl(\ge 1-\frac{P_{\cS_0,E\cap\sB(a,t)}^*(\cS_{B(a,rt)})}{P_{\cS_0,E\cap\sB(a,t)}^*(\cS_{B(a,Rt)})}
=\frac{P_{\cS_0,E\cap\sB(a,t)}^*(\cS_{B(a,Rt)})-P_{\cS_0,E\cap\sB(a,t)}^*(\cS_{B(a,rt)})}{P_{\cS_0,E\cap\sB(a,t)}^*(\cS_{B(a,Rt)})}
\\
=\frac{-\rho(\cS_0,\cS_{B(a,Rt)})+\rho(\cS_0,\cS_{B(a,rt)})}{P_{\cS_0,E\cap\sB(a,t)}^*(\cS_{B(a,Rt)})}
=\frac{\rho(\cS_{B(a,Rt)},\cS_{B(a,rt)})}{P_{\cS_0,E\cap\sB(a,t)}^*(\cS_{B(a,Rt)})}\biggr)
\\
\ge\frac{\log(R/r)}{2\log\bigl(\diam(\cS_g\wedge\cS_0)/\diam(\cS_0)\bigr)+\log\bigl(R/\bigl(\Capa_{\cS_0}(E\cap\sB(a,t))/t\bigr)\bigr)}.\label{eq:Choquetlower}
\end{multline}

\begin{remark}
The quantity $c(t)$ above is nothing but (a non-archimedean version of)
the Choquet capacity 
$C(E\cap\sB(a,t);U_{\overrightarrow{\cS_{B(a,Rt)}a}},\sB(a,rt))$
of $E\cap\sB(a,t)$
with respect to $\partial(\sB(a,rt))=\{\cS_{B(a,rt)}\}$
(see, e.g., \cite[Introduction]{Siciak97}).
\end{remark}

\subsection*{(ii)} Fix $R>r>1$ (so $r/R\in(0,1)$), and for every $n\in\bN$,
set
\begin{gather*}
 t_n:=\Bigl(\frac{r}{R}\Bigr)^n.
\end{gather*}
Then for every $a\in E\cap K$, 
every $n\in\bN$, and every $m\in\bN\cup\{0\}$, by an induction on $m$,
we have
\begin{gather}
h\bigl(\cdot;\sB(a,t_n)\cap E,U_{\overrightarrow{\cS_{B(a,Rt_n)}a}}\bigr)
\le e^{-\sum_{j=0}^m c(t_{n+j})}
\quad\text{on }U_{\overrightarrow{\cS_{B(a,rt_{n+m})}a}};\label{eq:Siciak}
\end{gather} 
indeed, for any subharmonic function
$u$ on $U_{\overrightarrow{\cS_{B(a,Rt_n)}a}}$ satisfying
$u|(\sB(a,t_n)\cap E)\le 0$ and $u<1$ on $U_{\overrightarrow{\cS_{B(a,Rt_n)}a}}$,
we have $u=1-(1-u)\le 1-c(t_n)\le e^{-c(t_n)}$ on 
$U_{\overrightarrow{\cS_{B(a,rt_n)}a}}=U_{\overrightarrow{\cS_{B(a,Rt_{n+1})}a}}$,
which yields not only
\begin{gather*}
 h(\cdot;\sB(a,t_n)\cap E,U_{\overrightarrow{\cS_{B(a,Rt_n)}a}})\le e^{-c(t_n)}
\quad\text{on }U_{\overrightarrow{\cS_{B(a,rt_n)}a}}=U_{\overrightarrow{\cS_{B(a,Rt_{n+1})}a}} 
\end{gather*}
but, together with $t_{n+1}\le t_n$, also
\begin{multline*}
\frac{h\bigl(\cdot;\sB(a,t_n)\cap E,U_{\overrightarrow{\cS_{B(a,Rt_n)}a}}\bigr)}{e^{-c(t_n)}}
\Bigl(\le h\bigl(\cdot;\sB(a,t_n)\cap E,U_{\overrightarrow{\cS_{B(a,Rt_{n+1})}a}}\bigr)\Bigr)\\
\le h\bigl(\cdot;\sB(a,t_{n+1})\cap E,U_{\overrightarrow{\cS_{B(a,Rt_{n+1})}a}}\bigr)\quad\text{on }U_{\overrightarrow{\cS_{B(a,Rt_{n+1})}a}}.
\end{multline*}
Now an induction concludes the desired estimate (see \cite[p.\ 212]{Siciak97}).

\subsection*{(iii)} 
Also fix $n\in\bN$ so large that $Rt_n<\diam(\cS_0)(\le\diam(\cS_0\wedge a)$
for every $a\in E\cap K$); then
for every $a\in E\cap K$ and every $m\in\bN\cup\{0\}$, the latter half 
\begin{gather*}
 \cS_{B(a,Rt_{n+m})}\in[\cS_0,a]
\end{gather*}
in the admissibility \eqref{eq:admissibility} for $t=t_{n+m}$ holds.
 
To any $\delta\in(0,\min\{1,rt_{n+1}\}]$, let us associate $m_\delta\in\bN$ 
such that
\begin{gather*}
 \delta\in[rt_{n+m_\delta+1},rt_{n+m_\delta}],
\end{gather*}
the lower bound of $\delta$ in which is equivalent to the inequality
\begin{gather*}
-(m_\delta+1)\log(R/r)\le-\log\bigl((r/R)^nr\bigr)+\log\delta.
\end{gather*}
\subsection*{(iv)} Suppose now that the compact subset $E$ (in $\sA^1$ 
under the normalization in {\bfseries (i)})
has the $($non-infinitesimal$)$ lower capacity density property,
for some $c\in(0,1)$.
Then using the latter half of the uniform comparison \eqref{eq:unifcompare}
between the kernel functions $|\cS-\cS'|_\infty$ and $[\cS,\cS']_{\cS_0}$ 
on $E$, we have
\begin{multline*}
 c_0:=\inf_{\cS\in E}\inf_{t\in(\diam\cS,\diam_\infty E)}\frac{\Capa_{\cS_0}(E\cap\sB(\cS,t))}{t}\\
\Bigl(\ge\inf_{\cS\in E}\inf_{t\in(\diam\cS,\diam_\infty E)}\frac{\bigl(\inf_E[\cdot,\infty]_g\bigr)^2\Capa_{\infty}(E\cap\sB(\cS,t))}{t}\Bigr)\\
\ge c\cdot\bigl(\inf_E[\cdot,\infty]_g\bigr)^2
\in(0,1), 
\end{multline*}
and set
\begin{multline*}
 \ell=\ell_E:=\frac{1}{2\log\bigl(\diam(\cS_g\wedge\cS_0))/\diam(\cS_0)\bigr)+\log\bigl(R/c_0)}\in\Bigl(0,\frac{1}{\log(R/c_0)}\Bigr).
\end{multline*}
Then for every $a\in E\cap K$, 
using \eqref{eq:Siciak} and \eqref{eq:Choquetlower}, we have
\begin{multline}
h\bigl(\cdot;E\cap\sB(a,t_n),U_{\overrightarrow{\cS_{B(a,t_nR)}a}}\bigr)
\le\exp\Bigl(-\sum_{j=0}^{m_\delta}\ell\log(R/r)\Bigr)\\
=e^{-\ell(m_\delta+1)\log(R/r)}
 \le \biggl(\frac{1}{(R/r)^nr}\biggr)^{\ell}\cdot\delta^{\ell}
\quad\text{on }\sB(a,\delta)(\subset U_{\overrightarrow{\cS_{B(a,rt_{n+m_{\delta}})}a}}).\label{eq:conti}
\end{multline}

\subsection*{(v)} Let us complete the proof of Theorem \ref{th:holder}.
By \eqref{eq:conti} and the former half in \eqref{eq:compare},
for any $a\in E\cap K$, $P_{\cS_0,E\cap\sB(a,t_n)}$ is continuous at $a$,
and then so is $P_{\cS_0,E}$ since $P_{\cS_0,E}\le P_{\cS_0,E\cap\sB(a,t_n)}$ on $\sP^1$. Then recalling the definition of $\ell=\ell_E$,
by \eqref{eq:Greenlower},
the maximum principle for subharmonic functions
(applied to $P_{\cS_0,E\cap\sB(a,t_n)}^*$,
which is $\ge P_{\cS_0,E\cap\sB(a,t_n)}$,
on $U_{\overrightarrow{\cS_{B(a,t_nR)}a}}$), and
the former half in \eqref{eq:compare} again, we have
\begin{multline*}
 \ell\cdot\bigl(p_{\cS_0,\nu_{\cS_0,E}}-\log\Capa_{\cS_0}(E)\bigr)
\Bigl(\le\frac{P_{\cS_0,E}}{P_{\cS_0,E\cap\sB(a,t_n)}^*(\cS_{B(a,Rt_n)})}\Bigr)\\
\le h\bigl(\cdot;E\cap\sB(a,t_n),U_{\overrightarrow{\cS_{B(a,t_nR)}a}}\bigr)
\quad\text{on }\sB(a,\delta)(\subset U_{\overrightarrow{\cS_{B(a,Rt_n)}a}}). 
\end{multline*}
This with \eqref{eq:conti} completes the proof, recalling 
the latter half of the 
uniform comparison \eqref{eq:unifcompare} between the kernel functions 
$|\cS-\cS'|_\infty$ and $[\cS,\cS']_g$ on $E$ again.
\end{proof}

\begin{acknowledgement}
The author was partially supported by JSPS Grant-in-Aid 
for Scientific Research (C), 19K03541 and (B), 19H01798.
\end{acknowledgement}

\def\cprime{$'$}

\end{document}